\documentclass{amsart}
\usepackage{mathscinet,amssymb,latexsym,stmaryrd,phonetic}
\usepackage{enumerate,mathrsfs,hyperref,cmll,color}

\newtheorem{theorem}{Theorem}[section]
\newtheorem{proposition}[theorem]{Proposition}
\newtheorem{lemma}[theorem]{Lemma}

\newtheorem{question}[theorem]{Question}
\newtheorem{corollary}[theorem]{Corollary}

\theoremstyle{definition}

\newtheorem{claim}[theorem]{Claim}
\newtheorem{fact}[theorem]{Fact}

\theoremstyle{remark}
\newtheorem{remark}[theorem]{Remark}

\numberwithin{equation}{section}

\newcommand{\ZFC}{\mathsf{ZFC}}
\newcommand{\CH}{\mathsf{CH}}

\newcommand{\restr}[2]{#1\upharpoonright {#2}}
\newenvironment{proofclaim}[1][Proof of the claim]{\textbf{#1.} }{\hfill \rule{0.5em}{0.5em}}
\newenvironment{prooffact}[1][Proof of the fact]{\textbf{#1.} }{\hfill \rule{0.5em}{0.5em}}
\newenvironment{proofmain}[2][Proof of the Theorem]{\textbf{#1 {#2}.} }{\hfill \rule{0.5em}{0.5em}}

\begin{document}
\title[Countably compact groups]{Countably compact groups without non-trivial convergent sequences}


%    author one information
\author[Hru\v{s}\'{a}k]{M. Hru\v{s}\'{a}k}
\address{Centro de Ciencias Matem\'aticas\\ Universidad Nacional Aut\'onoma de M\'exico\\ Campus Morelia\\Morelia, Michoac\'an\\ M\'exico 58089}
\curraddr{}
\email{michael@matmor.unam.mx}
\thanks{{The research of the first author was supported  by a PAPIIT grant IN100317 and CONACyT grant A1-S-16164. The third named author was partially supported by the PAPIIT grants  IA100517 and IN104419. Research of the fourth author was partially supported by European Research Council grant 338821. Paper 1173 on the fourth author's list}}
\urladdr{http://www.matmor.unam.mx/~michael}

\author[van Mill]{J. van Mill}
\address{KdV Institute for Mathematics, University of Amsterdam, Science Park 105-107, P.O. Box 94248, 1090 GE Amsterdam, The Netherlands}
\curraddr{}
\email{j.vanMill@uva.nl}
\thanks{}

\author[Ramos-Garc\'{\i}a]{U. A. Ramos-Garc\'{\i}a}
\address{Centro de Ciencias Matem\'aticas\\ Universidad Nacional Aut\'onoma de M\'exico\\ Campus Morelia\\Morelia, Michoac\'an\\ M\'exico 58089}
\curraddr{}
\email{ariet@matmor.unam.mx}
\thanks{}

\author[Shelah]{S. Shelah}
\address{Einstein Institute of Mathematics, Edmond J. Safra Campus, The Hebrew University of Jerusalem, Givat Ram, Jerusalem, 91904, Israel and Department of Mathematics, Hill Center - Busch Campus, Rutgers, The State University of New Jersey, 110 Frelinghuysen Road, Piscataway, NJ 08854-8019, USA}
\curraddr{}
\email{shelah@math.huji.ac.il}
\thanks{}
\urladdr{http://shelah.logic.at}

%    \subjclass is required.

\date{\today}


%    Abstract is required.

\begin{abstract}
We construct, in {\sf ZFC}, a countably compact subgroup of $2^\mathfrak c$ without non-trivial convergent sequences, answering an old problem of van Douwen. As a consequence we also prove the existence of two countably compact groups $\mathbb G_0$ and  $\mathbb G_1$ such that the product  $\mathbb G_0 \times \mathbb G_1$ is not countably compact, thus answering a classical problem of Comfort.
\end{abstract}

\maketitle

\section{ Introduction}
The celebrated Comfort-Ross theorem \cite{MR207886,MR776643} states that any product of pseudo-compact topological groups is pseudo-compact, in stark contrast with the examples due to Nov\'ak \cite{MR0060212} and Terasaka \cite{MR0051500} who constructed pairs of countably compact spaces whose product is not even pseudo-compact. This motivated Comfort \cite{Letter-Comfort} (repeated in \cite{MR1078657}) to ask:

\begin{question}[Comfort \cite{MR1078657}] Are there countably compact groups $\mathbb G_0, \mathbb G_1$ such that $\mathbb G_0\times\mathbb G_1$ is not countably compact?
\end{question}

The first consistent positive answer was given by van Douwen \cite{MR586725} under {\sf MA}, followed by  Hart-van Mill \cite{MR982236} under {\sf MA${}_{ctble}$}. In his paper  van Douwen showed that
every Boolean countably compact group without non-trivial convergent sequences contains two countably compact subgroups whose product is not countably compact, and asked:

\begin{question}[van Douwen \cite{MR586725}] Is there  a countably compact group without non-trivial convergent sequences?
\end{question}

In fact, the first example of such a group was constructed by Hajnal and Juh\'asz \cite{MR0431086} a few years before van Douwen's \cite{MR586725} assuming {\sf CH}. Recall, that every compact topological group contains a non-trivial convergent sequence, as an easy consequence of the classical and highly non-trivial theorem of Ivanovski\u{\i}-Vilenkin-Kuz'minov  (see \cite{MR0104753}) that every compact topological group is \emph{dyadic}, \emph{i.e.,} a continuous image of $2^{\kappa}$ for some cardinal number $\kappa$.

\smallskip

Both questions have been studied extensively in recent decades, providing a large variety of sufficient conditions for the existence of examples to  these questions, much work being done by Tomita and collaborators \cite{MR2029279,MR2113947,MR1925707,MR2921841,MR2519216,MR1426926,MR1664516,MR1974663,MR2107169,MR2163099,MR2133678,MR2080287}, but also others \cite{MR2284939,MR1719996,MR2139740,MR792239,MR1083312}. The questions are considered central in the theory of topological groups \cite{MR3241473,MR2433295,MR776643,MR1078657,dikranjan2007selected,MR1666795,MR1900269}.

\smallskip

Here we settle both problems by constructing in {\sf ZFC} a countably compact subgroup of $2^\mathfrak c$ without non-trivial convergent sequences. 

\smallskip

The paper is organized as follows: In Section 2 we fix notation and review basic facts concerning ultrapowers, Fubini products of ultrafilters and Bohr topology. In Section 3 we study van Douwen's problem in the realm of 
$p$-compact groups. We show how iterated ultrapowers can be used to give interesting partial solutions to the problem. In particular, we show that an iterated ultrapower of the countable Boolean group endowed with the Bohr topology via a selective ultrafilter $p$ produces a $p$-compact subgroup of $2^{\mathfrak{c}}$ without non-trivial convergent sequences. This on the one hand raises interesting questions about ultrafilters, and on the other hand serves as a warm up for Section 4, where the main result of the paper is proved by constructing
a countably compact subgroup of $2^{\mathfrak{c}}$ without non-trivial convergent sequences using not a single ultrafilter, but rather a carefully constructed $\mathfrak{c}$-sized family of ultrafilters.

\section{Notation and terminology}

Recall that an infinite topological space $X$ is \emph{countably compact} if every infinite subset of $X$ has an accumulation point. Given $p$ a nonprincipal ultrafilter on $\omega$ (for short, $p\in \omega^*$), a point $x\in X$ and a sequence
$\{x_n: n\in\omega\}\subseteq X$ we say (following \cite{MR0251697})  that  $x=p$-$\lim_{n \in \omega} x_n$  if for every open $U\subseteq X$ containing $x$ the set $\{n \in \omega \colon x_n\in U\}\in p$. It follows that a space $X$ is countably compact if and only if every sequence $\{x_n: n\in\omega\}\subseteq X$ has a $p$-limit in $X$ for some ultrafilter $p\in\omega^*$. Given an ultrafilter  $p\in\omega^*$, a space $X$ is \emph{$p$-compact}
if for every sequence $\{x_n: n\in\omega\}\subseteq X$ there is an $x\in X$ such that $x=p$-$\lim_{n \in \omega} x_n$.

\medskip 

For introducing the following definition, we fix a bijection $\varphi: \omega\to\omega\times\omega$, and for a limit ordinal $\alpha<\omega_1$, we pick an increasing sequence $\{\alpha_n: n\in\omega\}$ of smaller ordinals with supremum $\alpha$.
Given an ultrafilter $p\in\omega^*$, the \emph{iterated Fubini powers} or \emph{Frol\'{\i}k sums} \cite{MR0203676} of $p$ are defined recursively as follows:
$$ p^1=p$$
$$p^{\alpha+1} =\{A\subseteq \omega: \{n:\{m: (n,m)\in \varphi(A)\}\in p^{\alpha}\}\in p\}\text{ and }$$
$$p^{\alpha} =\{A\subseteq \omega: \{n:\{m: (n,m)\in \varphi(A)\}\in p^{\alpha_n}\}\in p\}\text{ for } \alpha \text{ limit.}$$
The choice of the ultrafilter $p^\alpha$ depends on (the arbitrary) choice of $\varphi$ and the choice of the sequence $\{\alpha_n: n\in\omega\}$, however, the \emph{type} of $p^\alpha$ does not (see \emph{e.g.}, \cite{MR0203676,MR1227550}).

\medskip
For our purposes we give an alternative definition of the iterated Fubini powers of $p$: given $\alpha < \omega_{1}$ we fix a well-founded tree $T_{\alpha} \subset \omega^{<\omega}$ such that 

\begin{enumerate}[(i)]
    \item $\rho_{T_{\alpha}}(\varnothing)=\alpha$, where $\rho_{T_{\alpha}}$ denotes the rank function on $\langle T_{\alpha},\subseteq\rangle$;
    \item For every $t \in T_{\alpha}$, if $\rho_{T_{\alpha}}(t) > 0$ then $t^{\frown}n \in T_{\alpha}$ for all $n \in \omega$.
\end{enumerate}

For $\beta \leqslant \alpha$, let 
$\Omega_{\beta}(T_{\alpha})=\{t \in T_{\alpha} \colon \rho_{T_{\alpha}}(t)=\beta\}$ and 
$T_{\alpha}^{+}=\{t \in T_{\alpha} \colon \rho_{T_{\alpha}}(t) > 0\}$.

\medskip

If $p \in \omega^{*}$, then $\mathbb{L}_{p}(T_{\alpha})$ will be used to denote the collection of all trees $T \subseteq T_{\alpha}$ such that for every $t \in T \cap T_{\alpha}^{+}$ the set $\text{succ}_{T}(t)=\{n \in \omega \colon t^{\frown}n \in T\}$ belongs to $p$. Notice that each $T \in \mathbb{L}_{p}(T_{\alpha})$ is also a well-founded tree with $\rho_{T}(\varnothing)=\alpha$. Moreover, the family $\{\Omega_{0}(T) \colon T \in \mathbb{L}_{p}(T_{\alpha})\}$ forms a base of an ultrafilter on $\Omega_{0}(T_{\alpha})$ 
which has the same type of $p^{\alpha}$. If $T \in \mathbb{L}_{p}(T_{\alpha})$ and $U \in p$, $\restr{T}{U}$ denotes  the tree in $\mathbb{L}_{p}(T_{\alpha})$ for which 
$\text{succ}_{\restr{T}{U}}(t)= \text{succ}_{T}(t) \cap U$ for all 
$t \in (\restr{T}{U})^{+}$.

\medskip

Next we recall the \emph{ultrapower} construction from model theory and algebra. Given a group $\mathbb G$ and an ultrafilter $p\in\omega^*$, denote by
$$\mathsf{ult}_p(\mathbb G)=\mathbb G^\omega/\equiv \text{, where } f\equiv g \text{ iff } \{n: f(n)=g(n)\}\in p.$$
The \emph{Theorem of \L\'os} \cite{MR0075156} states that for any formula $\phi$ with parameters $[f_0], [f_1],\dots$ $[f_n]$, $\mathsf{ult}_p(\mathbb G)\models \phi ([f_0], [f_1],\dots [f_n])$ if and only if
$\{k : \mathbb G\models\phi (f_0(k), f_1(k),$ $\dots$ $f_n(k))\}\in p$. In particular, $\mathsf{ult}_p(\mathbb G)$ is a group with the same first order properties as $\mathbb G$. 

\medskip
%\emph{i.e.}, if $\mathbb G$ is boolean then so is $ult_p(\mathbb G)$.

 There is a natural embedding of $\mathbb G$ into $\mathsf{ult}_p(\mathbb G)$ sending each $g\in \mathbb G$ to the equivalence class of the constant function with value $g$. We shall therefore consider $\mathbb G$ as a subgroup of $\mathsf{ult}_p(\mathbb G)$. Also, without loss of generality, we can assume that $\text{dom}(f) \in p$ for every $[f] \in \mathsf{ult}_p(\mathbb G)$.

\medskip

Recall that the \emph{Bohr topology} on a group $\mathbb{G}$ is the weakest group topology making every homomorphism $\Phi \in \text{Hom}(\mathbb{G},\mathbb{T})$ continuous, where the circle group $\mathbb{T}$ carries the usual compact topology. We let $(\mathbb{G},\tau_{\, \text{Bohr}})$ denote $\mathbb{G}$ equipped with the Bohr topology.

\medskip

Finally, our set-theoretic notation is mostly standard and follows \cite{MR597342}. In particular, recall that an ultrafilter $p \in \omega^{*}$ is a \emph{$P$-point} if 
every function on $\omega$ is finite-to-one or constant when restricted to some set in the ultrafilter and, an ultrafilter $p \in \omega^{*}$ is a \emph{$Q$-point} if every finite-to-one function on $\omega$ becomes one-to-one when restricted to a suitable set in the ultrafilter. The ultrafilters $p \in \omega^{*}$ which are P-point and Q-point are called \emph{selective} ultrafilters. For more background on set-theoretic aspects of ultrafilters see  \cite{MR2757533}.

\medskip

\section{Iterated ultrapowers as \texorpdfstring{$p$}{Lg}-compact groups}\label{p-comp}

In this section we shall give a canonical construction of a $p$-compact group  
for every ultrafilter $p \in \omega^{*}$. This will be done by studying the iterated ultrapower construction.

\medskip

Fix a group $\mathbb{G}$ and put $\mathsf{ult}_{p}^{0}(\mathbb{G})=\mathbb{G}$. Given an ordinal $\alpha$ with $\alpha > 0$, let 
\[
\mathsf{ult}_{p}^{\alpha}(\mathbb{G})=\mathsf{ult}_{p}\left(\varinjlim_{\beta < \alpha}\mathsf{ult}_{p}^{\beta}(\mathbb{G}))\right),
\]
where $\varinjlim_{\beta < \alpha}\mathsf{ult}_{p}^{\beta}(\mathbb{G})$ denotes the direct limit of the direct system $\langle \mathsf{ult}_{p}^{\beta}(\mathbb{G}), \varphi_{\delta\beta} \colon \delta \leqslant \beta < \alpha \rangle$ with the following properties:
\begin{enumerate}
    \item $\varphi_{\delta\delta}$ is the identity function on $\mathsf{ult}_{p}^{\delta}(\mathbb{G})$, and
    \item $\varphi_{\delta\beta}\colon \mathsf{ult}_{p}^{\delta}(\mathbb{G}) \to \mathsf{ult}_{p}^{\beta}(\mathbb{G})$ is the canonical embedding of $\mathsf{ult}_{p}^{\delta}(\mathbb{G})$ into $\mathsf{ult}_{p}^{\beta}(\mathbb{G})$, defined recursively by
    $\varphi_{\delta, \alpha +1}([f])=$ the constant function with value $[f]$, and $\varphi_{\delta, \alpha }([f])=$ the direct limit of $\varphi_{\delta, \beta }([f]), \ \beta<\alpha$ for a limit ordinal $\alpha$.
\end{enumerate}

In what follows, we will abbreviate $\mathsf{ult}_{p}^{\alpha^{-}}(\mathbb{G})$ for $\varinjlim_{\beta < \alpha}\mathsf{ult}_{p}^{\beta}(\mathbb{G})$. Moreover, we will treat  $\mathsf{ult}_{p}^{\alpha^{-}}(\mathbb{G})$ as 
$\bigcup_{\beta < \alpha}\mathsf{ult}_{p}^{\beta}(\mathbb{G})$ and, in such case, we put  
$
\text{ht}(a)=\min\{\beta < \alpha \colon a \in \mathsf{ult}_{p}^{\beta}(\mathbb{G})\}
$ 
for every $a \in \mathsf{ult}_{p}^{\alpha^{-}}(\mathbb{G})$. This is, of course, formally wrong, but is facilitated by our identification of $\mathbb{G}$ with a subgroup of $\mathsf{ult}_{p}(\mathbb{G})$. In this way we can avoid talking about direct limit constructions.

\medskip

We now consider $(\mathbb{G},\tau_{\, \text{Bohr}})$. Having fixed an ultrafilter $p \in \omega^{*}$, this topology naturally \emph{lifts} to a topology on $\mathsf{ult}_{p}(\mathbb{G})$ as follows: Every $\Phi \in \text{Hom}(\mathbb{G},\mathbb{T})$ naturally extends to a homomorphism $\overline{\Phi} \in \text{Hom}(\mathsf{ult}_{p}(\mathbb{G}),\mathbb{T})$ by letting 

\begin{equation}\label{e:p-lim} 
    \overline{\Phi}([f])=p\text{ -}\lim_{n \in \omega}\Phi(f(n)).
\end{equation}

By {\L}\'os's theorem, $\overline \Phi$ is indeed a homomorphism from $\mathsf{ult}_p(\mathbb{G})$ to $\mathbb{T}$ and hence the weakest topology making every $\overline{\Phi}$ continuous, where $\Phi \in \text{Hom}(\mathbb{G},\mathbb{T})$, is a group topology on $\mathsf{ult}_p(\mathbb{G})$. This topology will be denoted by $\tau_{\, \overline{\text{Bohr}}}$.

\medskip

The following is a trivial, yet fundamental fact:

\begin{lemma}\label{p-lim} For every $f:\omega\to \mathbb{G}$, $[f]=p$-$\lim_{n \in \omega} f(n)$ in $\tau_{\, \overline{\text{Bohr}}}$.
\end{lemma}

\begin{proof} This follows directly from the definition of $\overline \Phi$ and the identification of $\mathbb G$ with a subgroup of $\mathsf{ult}_p(\mathbb{G})$.
\end{proof}

The group that will be relevant for us is the group $\mathsf{ult}_{p}^{\omega_{1}}(\mathbb{G})$, endowed with the topology 
$\tau_{\, \overline{\text{Bohr}}}$ induced by the homomorphisms in $\text{Hom}(\mathbb{G},\mathbb{T})$ extended recursively all the way to  $\mathsf{ult}_p^{\omega_{1}}(\mathbb{G})$ by the same formula (\ref{e:p-lim}).

\medskip

The (iterated) ultrapower with this topology is usually not Hausdorff (see \cite{MR2207497,MR698308}), so we identify the inseparable functions and denote by $(\mathsf{Ult}_{p}^{\omega_{1}}(\mathbb{G}),\tau_{\, \overline{\text{Bohr}}})$ this quotient. More explicitly,  
\[
\mathsf{Ult}_{p}^{\omega_{1}}(\mathbb{G}) = \mathsf{ult}_{p}^{\omega_{1}}(\mathbb{G})/K,
\]
where $K=\bigcap_{\Phi \in \text{Hom}(\mathbb{G},\mathbb{T})}\text{Ker}(\overline{\Phi})$. The natural projection will be denoted by 
\[
\pi \colon \mathsf{ult}_{p}^{\omega_{1}}(\mathbb{G}) \to \mathsf{ult}_{p}^{\omega_{1}}(\mathbb{G})/K.
\]

\medskip 

The main reason for considering the iterated Fubini powers here is the following simple and crucial fact:

\begin{proposition}\label{p-compact} 
Let  $p\in\omega^*$ be an ultrafilter.
\begin{enumerate}[\upshape (1)]
\item $\mathsf{ult}_{p}^{\alpha}(\mathbb{G}) \simeq \mathsf{ult}_{p^{\alpha}}(\mathbb{G})$ for $\alpha < \omega_{1}$, and
\item $(\mathsf{Ult}_{p}^{\omega_{1}}(\mathbb{G}),\tau_{\, \overline{\text{Bohr}}})$ is a Hausdorff $p$-compact topological group.
\end{enumerate}
\end{proposition}

\begin{proof} 
To prove (1), fix an $\alpha < \omega_{1}$. For given $[f] \in \mathsf{ult}_{p}^{\alpha}(\mathbb{G})$, recursively define a tree $T_{f} \in \mathbb{L}_{p}(T_{\alpha})$ and a function $\hat{f} \colon T_{f} \to \mathsf{ult}_{p}^{\alpha}(\mathbb{G})$ so that 

\begin{itemize}
    \item $\text{succ}_{T_{f}}(\varnothing)=\text{dom}(f_{\varnothing})$ and $\hat{f}(\varnothing)=[f_{\varnothing}]$, where $f_{\varnothing}=f$;
    \item if $\hat{f}(t)$ is defined say $\hat{f}(t)=[f_{t}]$, then $\text{succ}_{T_{f}}(t)=\text{dom}(f_{t})$ and 
    $\hat{f}(t^{\frown}n) = f_{t}(n)$ for every $n \in \text{succ}_{T_{f}}(t)$.
\end{itemize}

\medskip

We define $\varphi \colon \mathsf{ult}_{p}^{\alpha}(\mathbb{G}) \to \mathsf{ult}_{p^{\alpha}}(\mathbb{G})$ given by 

\[
\varphi([f])= [\restr{\hat{f}}{\Omega_{0}(T_{f})}].
\]

\begin{claim}
$\varphi$ is an isomorphism.
\end{claim}

\begin{proofclaim}
To see that $\varphi$ is a surjection, let $[f] \in \mathsf{ult}_{p^{\alpha}}(\mathbb{G})$ be such that $\text{dom}(f) = \Omega_{0}(T_{f})$ for some $T_{f} \in \mathbb{L}_{p}(T_{\alpha})$. Consider the function $\check{f} \colon T_{f} \to \mathsf{ult}_{p}^{\alpha}(\mathbb{G})$ defined recursively by 

\begin{itemize}
    \item $\restr{\check{f}}{\Omega_{0}(T_{f})}=f$ and, 
    \item if $ t \in T_{\alpha}^{+}$, then $\check{f}(t)=[\langle \check{f}(t^{\frown}n) \colon n \in \text{succ}_{T_{f}}(t)\rangle]$.
\end{itemize}

\medskip

Notice that the function $\check{f}$ satisfies that $\check{f}(t) \in \mathsf{ult}_{p}^{\rho_{T_{f}(t)}}(\mathbb{G})$ for every $t \in T_{f}$. In particular, $\check{f} (\varnothing) \in \mathsf{ult}_{p}^{\alpha}(\mathbb{G})$ and, a routine calculation shows that $\varphi(\check{f} (\varnothing)) = [f]$.

\medskip 

To see that $\varphi$ is injective, suppose that $\varphi([f])=\varphi([g])$. Then there exists a tree $T \in \mathbb{L}_{p}(T_{\alpha})$ such that 
\[
\restr{\hat{f}}{\Omega_{0}(T)} = \restr{\hat{g}}{\Omega_{0}(T)}.
\]
If set $h := \restr{\hat{f}}{\Omega_{0}(T)}$, then we can verify recursively that $\check{h}(\varnothing)=[f]=[g]$. Therefore, $\varphi$ is a one-to-one function.

\medskip

Finally, using again a recursive argument, one can check that $\varphi$ preserves the group structure.
\end{proofclaim}

\medskip

To prove (2) note that by definition $\mathsf{Ult}_{p}^{\omega_{1}}(\mathbb{G})$ is a Hausdorff topological group. To see that $\mathsf{Ult}_{p}^{\omega_{1}}(\mathbb{G})$ is $p$-compact, since $\mathsf{Ult}_{p}^{\omega_{1}}(\mathbb{G})$ is a continuous image of $\mathsf{ult}_{p}^{\omega_{1}}(\mathbb{G})$, it suffices to check that $\mathsf{ult}_{p}^{\omega_{1}}(\mathbb{G})$ is $p$-compact. Let 
$f \colon \omega \to \mathsf{ult}_{p}^{\omega_{1}}(\mathbb{G})$ be a sequence and let $n \in \omega$. So $f(n) \in \mathsf{ult}_{p}(\mathsf{ult}_{p}^{\omega_{1}^{-}}(\mathbb{G}))$, that is, there exists 
$f_{n} \colon \omega \to \bigcup_{\alpha < \omega_{1}}\mathsf{ult}_{p}^{\alpha}(\mathbb{G})$ such that $f(n)=[f_{n}]$. Thus, for every $n \in \omega$ there exists $\alpha_{n} < \omega_{1}$ such that $f(n) \in \mathsf{ult}_{p}^{\alpha_{n}}(\mathbb{G})$ and hence 
$[f] \in \mathsf{ult}_{p}^{\alpha}(\mathbb{G})$ for $\alpha =\sup \{\alpha_n: n\in \omega\} < \omega_{1}$. Then $[f]=p$-$\lim_{n \in \omega} f(n)$ in $\tau_{\, \overline{\text{Bohr}}}$ as by the construction
$\overline{\Phi}([f])=p$-$\lim \overline \Phi(f(n))$ for every $\Phi \in \text{Hom}(\mathbb{G},\mathbb{T})$.
This gives us the $p$-compactness of $\mathsf{ult}_{p}^{\omega_{1}}(\mathbb{G})$. 
\end{proof}

The plan for our construction is as follows: fix an ultrafilter $p \in \omega^{*}$, find a suitable topological group $\mathbb{G}$ without convergent sequences and consider 
$(\mathsf{Ult}_{p}^{\omega_{1}}(\mathbb{G}),\tau_{\, \overline{\text{Bohr}}})$. The remaining issue is: Does $(\mathsf{Ult}_{p}^{\omega_{1}}(\mathbb{G}),\tau_{\, \overline{\text{Bohr}}})$ have non-trivial convergent sequences?

\medskip

While our approach is applicable to an arbitrary group $\mathbb{G}$, in the remainder of this paper we will be dealing exclusively with \emph{Boolean} groups, \emph{i.e.,} groups where each element is its own inverse.\footnote{The general case will be dealt with in a separate paper.} These groups are, in every infinite cardinality $\kappa$,  isomorphic to the group $[\kappa]^{<\omega}$ with the symmetric difference $\triangle$ as the group operation and $\varnothing$ as the neutral element. Every Boolean group is a vector space over the trivial $2$-element field which we identify
with $2=\{0,1\}$. Hence, we can talk, \emph{e.g.,} about \emph{linearly independent} subsets of a Boolean group. Also, since every homomorphism from a Boolean group into the torus $\mathbb T$ takes at most two values (in the unique subgroup of $\mathbb T$ of size $2$) we may and will identify $\text{Hom}([\omega]^{<\omega}, \mathbb T)$ with
$\text{Hom}([\omega]^{<\omega}, 2)$ to highlight the fact that there are only two possible values. Hence also $\text{Hom}([\omega]^{<\omega}, 2)$ is a Boolean group and a vector space over the same field.

\medskip

The following theorem is the main result of this section.

\begin{theorem}\label{Th:SelectiveUltrafilter}
 Let $p \in \omega^{*}$ be a selective ultrafilter. Then 
 $(\mathsf{Ult}_{p}^{\omega_{1}}([\omega]^{<\omega}),\tau_{\, \overline{\text{Bohr}}})$ is a Hausdorff $p$-compact topological Boolean group without non-trivial convergent sequences.
\end{theorem}

In order to prove this theorem, we apply the first step of our plan.

\begin{proposition}\label{p:ground-group}
The group $[\omega]^{<\omega}$ endowed with the topology $\tau_{\, \text{Bohr}}$ is a non-discrete Hausdorff topological group without non-trivial convergent sequences.  
\end{proposition}

\begin{proof}
It is well-known and easy to see that $\tau_{\, \text{Bohr}}$ is a non-discrete Hausdorff group topology (\emph{e.g.,} see \cite{MR2433295} Section 9.9). To see that $\tau_{\, \text{Bohr}}$ has no non-trivial convergent sequences, assume that $f \colon \omega \to [\omega]^{<\omega}$ is a non-trivial sequence. Then $\text{rng}(f)$ is an infinite set.  Find an infinite linearly independent set $A \subseteq \text{rng}(f)$ and split it into two infinite pieces $A_{0}$ and $A_{1}$, and take 
$\Phi \in \text{Hom}([\omega]^{<\omega},2)$ such that $A_{i} \subseteq \Phi^{-1}(i)$ for every $i < 2$. Therefore, $\Phi$ is a witness that the sequence $f$ does not converge.
\end{proof}

We say that a sequence $\langle [f_{n}] \colon n \in \omega \rangle \subset \mathsf{ult}_{p}([\omega]^{<\omega})$ is \emph{$p$-separated} if for every $n \ne m \in \omega$ there is a $\Phi \in \text{Hom}([\omega]^{<\omega},2)$ such that $\overline{\Phi}([f_{n}]) \ne \overline{\Phi}([f_{m}])$. In other words, a sequence 
$\langle [f_{n}] \colon n \in \omega \rangle \subset \mathsf{ult}_{p}([\omega]^{<\omega})$ is $p$-separated if 
and only if its elements represent distinct elements of 
$$\mathsf{Ult}_{p}([\omega]^{<\omega})=\mathsf{ult}_{p}([\omega]^{<\omega})/K$$
where $K=\bigcap_{\Phi\in \text{Hom}([\omega]^{<\omega},2)} \text{Ker}(\overline{\Phi})$ and $\pi:\mathsf{ult}_{p}([\omega]^{<\omega})\to \mathsf{Ult}_{p}([\omega]^{<\omega}) $ is the corresponding projection.
\medskip

We next show that, in general, the plan does not work for all $p \in \omega^{*}$.

\begin{lemma}\label{L:ConvergentSequences}
 The following are equivalent:
 \begin{enumerate}[\upshape (1)]
     \item There exists a $p \in \omega^{*}$ such that $(\mathsf{Ult}_{p}([\omega]^{<\omega}),\tau_{\, \overline{\text{Bohr}}})$ has non-trivial convergent sequences.
     \item There exist a sequence $\langle \Phi_{n} \colon n \in \omega \rangle \subset 
 \text{Hom}([\omega]^{<\omega},2)$ and a mapping \\ $H \colon\text{Hom}([\omega]^{<\omega},2)\to \omega$ such that for every $n \in \omega$ the family 
 $$\{[\omega]^{<\omega}\setminus \text{Ker}(\Phi_{n})\} \cup \{\text{Ker}(\Phi)\colon \ H(\Phi) \leqslant n\}$$ 
 is centered.
 \end{enumerate}
\end{lemma}

\begin{proof}
 Let us prove (1) implies (2). Let $\tilde{f} \colon \omega \to \mathsf{Ult}_{p}([\omega]^{<\omega})$ be a non-trivial sequence, say $\tilde{f}(n)=\pi(f(n))$ ($n \in \omega$) where $f \colon \omega \to  \mathsf{Ult}_{p}([\omega]^{<\omega})$. Without loss of generality we can assume that $\tilde{f}$ is a one-to-one function converging to $\pi([\langle \varnothing\rangle])$, here $\langle \varnothing\rangle$ denotes the constant sequence where each term is $\varnothing$. So $\langle [f_{n}] \colon n \in \omega \rangle$ is a $p$-separated sequence $\tau_{\, \overline{\text{Bohr}}}$-converging to $[\langle \varnothing\rangle]$, where $[f_{n}]=f(n)$ for $n \in \omega$. By taking a subsequence if necessary, we may assume that for every $n \in \omega$ there is a 
 $\Phi_{n} \in \text{Hom}([\omega]^{<\omega},2)$ such that 
 $\overline{\Phi}_{n}([f_{n}])=1$. Now, by $\tau_{\, \overline{\text{Bohr}}}$-convergence of $\langle [f_{n}] \colon n \in \omega \rangle$, there is a mapping $H \colon \text{Hom}([\omega]^{<\omega},2) \to \omega$ such that for each $\Phi \in \text{Hom}([\omega]^{<\omega},2)$ and each $n \geqslant H(\Phi)$ it follows that $\overline{\Phi}([f_{n}])=0$. Now we will check that for every $n \in \omega$ the family $\{\text{Ker}(\Phi_{n})^{c}\} \cup \{\text{Ker}(\Phi)\colon$ $H(\Phi) \leqslant n\}$ is centered.\footnote{For a subset $A$ of the group $[\omega]^{<\omega}$, $A^{c}=[\omega]^{<\omega}\setminus A$.} For this, since $([\omega]^{<\omega},\tau_{\, \text{Bohr}})$ is without non-trivial convergent sequences and $[f_{n}] \xrightarrow{\tau_{\, \overline{\text{Bohr}}}} [\langle \varnothing\rangle]$, we may assume that 
 $[f_{n}]\ne [\langle a\rangle]$ for every $\langle n,a \rangle \in \omega \times [\omega]^{<\omega}$, that is, $f_{n}[U]$ is infinite for all 
 $\langle n,U\rangle \in \omega \times p$. Now, fix $n \in \omega$ and let $F \subset \text{Hom}([\omega]^{<\omega},2)$ be a finite set such that $H(\Phi) \leqslant n$ for every $\Phi \in F$. Then $\overline{\Phi}([f_{n}])=0$ for every $\Phi \in F$ and hence there exists $U_{F} \in p$ such that $\Phi(f_{n}(k))=0$ for every 
 $\langle k,\Phi\rangle \in U_{F}\times F$. Since $\overline{\Phi}_{n}([f_{n}])=1$, there exists $U_{n} \in p$ such that $\Phi_{n}(f_{n}(k))=1$ for every $k \in U_{n}$. Put $U=U_{F}\cap U_{n} \in p$. Then 
 $f_{n}[U] \subset \text{Ker}(\Phi_{n})^{c} \cap \bigcap_{\Phi \in F}\text{Ker}(\Phi)$, so we are done.

\medskip

 To prove (2) implies (1), first we observe that there is a sequence $\langle f_{n} \colon n \in \omega \rangle \subset ([\omega]^{<\omega})^{\omega}$ such that for each $F \in [\omega]^{<\omega}$ and every $\sigma \colon F \to [\omega]^{<\omega}$ there exists $k \in \omega$ such that $f_{i}(k) = \sigma(i)$ for all $i \in F$. Now, define $A_{\Phi,n}^{0}=\{k \in \omega \colon \Phi(f_{n}(k))=0\}$ and $A_{\Phi,n}^{1}=\{k \in \omega \colon \Phi(f_{n}(k))=1\}$ for all $(\Phi,n) \in \text{Hom}([\omega]^{<\omega},2)\times \omega$. 

\medskip

Fix  $\langle \Phi_{n} \colon n \in \omega \rangle \subset 
 \text{Hom}([\omega]^{<\omega},2)$ and  $H \colon\text{Hom}([\omega]^{<\omega},2)\to \omega$ as in (2).

 \begin{claim}\label{Claim:CenteredFamily}
 The collection $\bigcup_{n \in \omega}\{A_{\Phi_{n},n}^{1}\} \cup \{A_{\Phi,n}^{0} \colon H(\Phi)\leqslant n\}$ forms a centered family which generates a free filter $\mathcal{F}$.
\end{claim}

\begin{proofclaim}
 To show that such family is centered, let $m >0$ and for every $i<m$ fix a finite set $\{\Phi^{j} \colon j < m_{i}\} \subset H^{-1}[i+1]$. Then, considering all choice functions 
 \[
   \sigma \colon n \to \bigcup_{i<m}\left(\text{Ker}(\Phi_{i})^{c} \cap \bigcap_{j<m_{i}}  
   \text{Ker}(\Phi^{j})\right), 
 \]
 we can ensure that  
 \[
   \bigcap_{i < m}\left( A_{\Phi_{i},i}^{1} \cap \bigcap_{j < m_{i}}A_{\Phi^{j},i}^{0}\right)
 \]
 is an infinite set.
 
 To see that the filter $\mathcal{F}$ is free, let $k \in \omega$. If there is an $n \in \omega$ such that $f_{n}(k)=\varnothing$, then $k \notin A_{\Phi_{n}, n}^{1} \in \mathcal{F}$. In another case, since $\langle f_{n}(k) \colon n \in \omega\rangle$ does not $\tau_{\, \text{Bohr}}$-converge to $\varnothing$, there exists $\Phi \in \text{Hom}([\omega]^{<\omega},2)$ such that 
 $\Phi(f_{n}(k))=1$ for infinitely many $n$. Then pick one of such $n$ with $H(\Phi)\leqslant n$ and, $k \notin A_{\Phi,n}^{0} \in \mathcal{F}$.
\end{proofclaim}

\medskip

Let $p \in \omega^{*}$ extend $\mathcal{F}$. By Claim \ref{Claim:CenteredFamily}, it follows that 
\begin{enumerate}[(i)]
\item $\overline{\Phi}_{n}([f_{n}])=1$, for every $n \in \omega$.
\item The sequence $\langle \overline{\Phi}([f_{n}]) \colon n \in \omega\rangle$ converges to $0$, for every $\Phi \in \text{Hom}([\omega]^{<\omega},2)$, \emph{i.e.,} $\langle [f_{n}] \colon n \in \omega\rangle$ is a $\tau_{\, \overline{\text{Bohr}}}$-convergent sequence to $[\langle \varnothing\rangle]$.
\end{enumerate}
Finally, taking a subsequence if necessary, we can assume that 
$\langle [f_{n}] \colon n \in \omega\rangle$ is $p$-separated and, hence $\langle \pi([f_{n}]) \colon n \in \omega\rangle$ is a non-trivial convergent sequence in 
$(\mathsf{Ult}_{p}([\omega]^{<\omega}),\tau_{\, \overline{\text{Bohr}}})$.
\end{proof}

\medskip

\begin{remark}
Note that the filter $\mathcal{F}$ is actually an $F_{\sigma}$-filter.
\end{remark}

\medskip

\begin{theorem}
There exists a $p \in \omega^{*}$ such that $(\mathsf{Ult}_{p}([\omega]^{<\omega}),\tau_{\, \overline{\text{Bohr}}})$ has non-trivial convergent sequences.
\end{theorem}

\begin{proof}
We will show that the second clause of the Lemma \ref{L:ConvergentSequences} holds. To see this, choose any countable linearly independent set 
$\{\Phi_{n} \colon n \in \omega\} \subset \text{Hom}([\omega]^{<\omega},2)$. Let 
$W$ be a vector subspace of $\text{Hom}([\omega]^{<\omega},2)$ such that $\text{Hom}([\omega]^{<\omega},2)=\text{span}\{\Phi_{n} \colon n \in \omega\} \oplus W$. We define the mapping $H \colon\text{Hom}([\omega]^{<\omega},2)\to \omega$ as follows: 
\[
H(\Phi)=\min\{n \colon \Phi \in \text{span}\{\Phi_{i} \colon i<n\} \oplus W\}.
\]
Now, let $n \in \omega$ and fix a finite set $\{\Phi^{j} \colon j<m\} \subset H^{-1}[n+1]$. In order to show that 
\[
\text{Ker}(\Phi_{n})^{c} \cap \bigcap_{j<m} \text{Ker}(\Phi^{j})
\]
is infinite, we shall need a fact concerning linear functionals on a vector space.

\begin{fact}[\cite{MR1741419}, p. 124]
Let $V$ be a vector space and $\Phi, \Phi^{0}, \dots, \Phi^{m-1}$ linear functionals on $V$. Then the following statements are equivalent:
\begin{enumerate}[(1)]
    \item $\bigcap_{j<m} \text{Ker}(\Phi^{j}) \subset \text{Ker}(\Phi)$.
    \item $\Phi \in \text{span}\{\Phi^{j} \colon j<m\}$.\hfill $\square$
\end{enumerate}
\end{fact}

Using this fact, and noting that $\Phi_{n} \notin \text{span}\{\Phi^{j} \colon j<m\}$, one sees that 
\[
\text{Ker}(\Phi_{n})^{c} \cap \bigcap_{j<m} \text{Ker}(\Phi^{j}) \ne \emptyset.
\]
Pick an arbitrary $a \in \text{Ker}(\Phi_{n})^{c} \cap \bigcap_{j<m} \text{Ker}(\Phi^{j})$ and put
\[
K = \text{Ker}(\Phi_{n}) \cap \bigcap_{j<m} \text{Ker}(\Phi^{j}).
\]
Then $K$ is an infinite set, and hence $a + K$ is an infinite set too. But 
\[
a+K \subset \text{Ker}(\Phi_{n})^{c} \cap \bigcap_{j<m} \text{Ker}(\Phi^{j}),
\]
so we are done.
\end{proof}

\medskip

\begin{corollary}[$\CH$]\label{Cor:P-point}
There is a P-point $p \in \omega^{*}$ such that $(\mathsf{Ult}_{p}([\omega]^{<\omega}),\tau_{\, \overline{\text{Bohr}}})$ has non-trivial convergent sequences.
\end{corollary}

\begin{proof}
It is well-known (\emph{e.g.,} see \cite{MR3692233}) that assuming $\CH$ every $F_{\sigma}$-filter can be extended to a P-point.
\end{proof}

As $(\mathsf{Ult}_{p}([\omega]^{<\omega}),\tau_{\, \overline{\text{Bohr}}})$ is a topological subgroup of  $(\mathsf{Ult}_{p}^{\omega_1}([\omega]^{<\omega}),\tau_{\, \overline{\text{Bohr}}})$ there are ultrafilters (even P-points assuming {\sf CH}) such that $\mathsf{Ult}_{p}^{\omega_1}([\omega]^{<\omega})$ has a non-trivial convergent sequence.

\medskip

Selective ultrafilters and Q-points, have immediate combinatorial reformulations 
relevant in our context. Given a non-empty set $I$ and $\mathbb{G}$ a Boolean group, we shall call a set $\{ f_{i} \colon i \in I \}$ of functions $f_{i} \colon \omega \to \mathbb{G}$ \emph{$p$-independent} if 
\[
\left\{n \colon a + \sum_{i \in E}f_{i}(n) = \varnothing\right\} \notin p
\]
for every non-empty finite set $E \subset I$ and every $a \in \mathbb{G}$. Note that, in particular, a function $f \colon \omega \to \mathbb{G}$ is not constant on an element of $p$ if and only if $\{f\}$ is $p$-independent. Now, we will say that a function $f \colon I \to \mathbb{G}$ is \emph{linearly independent} if $f$ is one-to-one and $\{f(i) \colon i \in I\}$ is a linearly independent set and, a function $f \colon I \to \mathsf{ult}_{p}(\mathbb{G})$ is \emph{$p$-independent} if $f$ is one-to-one and $\{f_{i} \colon i \in I\}$ is a $p$-independent set, where $f(i)=[f_{i}]$ for $i \in I$.
\medskip

\begin{proposition}\label{selective} Let $p\in\omega^*$ be an ultrafilter. Then:
\begin{enumerate}[\upshape (1)]
\item $p$ is a Q-point if and only if for every finite-to-one function $f \colon \omega\to [\omega]^{<\omega}$ there is a set $U\in p$ such that $\restr{f}{U}$ is linearly independent.
\item The following are equivalent
\begin{enumerate}[\upshape (a)]
\item $p$ is selective;
\item for every function $f \colon \omega \to [\omega]^{<\omega}$ which is not constant on an element of $p$ there is a set $U\in p$ such that $\restr{f}{U}$ is linearly independent;
\item for every $p$-independent set $\{f_n \colon n\in\omega \}$ of functions $f_n \colon \omega\to [\omega]^{<\omega}$, there is a set $U \in p$ and a function $g \colon \omega \to \omega$ so that $\restr{f_{n}}{U \setminus g(n)}$  
is one-to-one for  $n \in \omega$, 
$f_{n}[U \setminus g(n)] \cap f_{m}[U \setminus g(m)] = \varnothing$ if $n \ne m$, and 
\[
\bigsqcup_{n \in \omega} f_{n}[U \setminus g(n)]
\]
is  linearly independent.\footnote{Here $\sqcup$ denotes the disjoint union.}
\end{enumerate}
\end{enumerate}
\end{proposition}

\begin{proof} 
Let us prove (1). Suppose first that $p$ is a Q-point. Let $f \colon \omega\to [\omega]^{<\omega}$ be a finite-to-one function. Recursively define a strictly increasing sequence $\langle n_{k} \colon k \in \omega\rangle$ of elements of $\omega$ and a strictly increasing sequence of finite subgroups $\langle H_{n} \colon n \in \omega\rangle$ of $[\omega]^{<\omega}$ so that 
\begin{enumerate}[\upshape (i)]
    \item $H_{n} \cap \text{rng}(f) \ne \emptyset$ for all $n \in \omega$, and
    \item $n_{k}=\max f^{-1}[H_{k}]\ \& \ f^{\prime\prime}[0,n_{k}] \subset H_{k+1}$, for all $k \in \omega$.
\end{enumerate}

Then partitioning $\omega$ into the union of even intervals, and the union of odd intervals, one of them is in $p$, say
\[
  A=\bigcup_{i \in \omega}[n_{2i},n_{2i+1}) \in p.
\]
Applying Q-pointness we can assume that there exists a $U \in p$ such that 
\[
  |[n_{2i},n_{2i+1}) \cap U|=1 \text{ for every } i \in \omega, 
\]
and $U \subseteq A$. By item (ii) and since $\langle H_{n} \colon n \in \omega\rangle$ is a strictly increasing sequence, it follows that $\restr{f}{U}$ is one-to-one and 
$\{f(n) \colon n\in U \}$ is linearly independent.

Suppose now that for every finite-to-one function $f \colon \omega\to [\omega]^{<\omega}$ there is a $U\in p$ such that $\restr{f}{U}$ is one-to-one and $\{f(n) \colon n\in U \}$ is linearly independent. Let $\langle I_{n} \colon n \in \omega\rangle$ be a partition of $\omega$ into finite sets. Define a finite-to-one function $f \colon \omega \to [\omega]^{<\omega}$ by putting $f(k)=\{n\}$ for each $k \in I_{n}$. Then there is an $U \in p$ such that 
$\restr{f}{U}$ is one-to-one and $\{f(n) \colon n\in U \}$ is linearly independent. Note that necessarily $|I_{n} \cap U| \leqslant 1$ for every $n \in \omega$ and therefore $p$ is a Q-point.  

\medskip

(2) To see (a) implies (b), let $f \colon \omega \to [\omega]^{<\omega}$ be a function which is not constant on an element of $p$. Using P-pointness, we may assume without loss of generality that $f$ is a finite-to-one function. So, by item (1), there is an $U \in p$ such that $\restr{f}{U}$ is one-to-one and $\{f(n) \colon n\in U \}$ is linearly independent.

\medskip

To see (b) implies (a), let $f \colon \omega \to [\omega]^{<\omega}$ be a function which is not constant on an element of $p$. By item (b), there is an $U \in p$ such that $\restr{f}{U}$ is one-to-one and $\{f(n) \colon n\in U \}$ is linearly independent, and hence $p$ is a P-point. To verify that $p$ is a Q-point, notice that every finite-to-one function $f \colon \omega \to [\omega]^{<\omega}$ is not constant on an element of $p$. Thus, by clause (1) we get the desired conclusion.

\medskip

To prove (a) implies (c), assume that $\{f_{n} \colon n \in \omega\}$ is a $p$-independent set of functions $f_n \colon \omega\to [\omega]^{<\omega}$. 

\begin{fact}\label{F:p-independent}
Given a finite $p$-independent set $\{f_{i} \colon i <n\}$, and a finite linearly independent set $A \subset [\omega]^{<\omega}$, the set of all $m \in \omega$ such that $A \sqcup \{f_{i}(m) \colon i <n\}$ is linearly independent, belongs to $p$.\hfill $\square$
\end{fact}

 Using Fact \ref{F:p-independent}, we can recursively construct a $p$-branching tree $T \subset \omega^{<\omega}$ such that for every $t \in T$, it follows that 
\[
\text{succ}_{T}(t)=\{m \colon A_{t} \sqcup \{f_{i}(m) \colon i \leqslant |t|\} \text{ is linearly independent}\},
\]
where $A_{t}=\{f_{i}(t(j)) \colon i < |t| \ \& \ j \in [i,|t|)\}$.

\medskip

By Galvin-Shelah's theorem (\cite[Theorem 4.5.3]{MR1350295}), let $x \in [T]$ be a branch such that $\text{rng}(x) \in p$. Thus, if we put $U = \text{rng}(x)$ and $g(n) = \max(\restr{x}{n})$ for $n \in \omega$, we get the properties  as in (c).

\medskip

Finally, notice that (b) is a particular instance of (c) when 
$\{f_{n} \colon n \in \omega\} = \{f\}$. Therefore, (c) implies (b).
\end{proof}

\medskip

\begin{remark}\label{R:selective}
In the previous theorem, it is possible to change the group $[\omega]^{<\omega}$ to any arbitrary Boolean group and, the conclusions of the theorem remain true.
\end{remark}

\medskip

For technical reasons, it will be necessary to reformulate the notion of $p$-indepen\-dence.

\begin{lemma}\label{L:p-independence}
Let $\mathbb{G}$ be a Boolean group and $0 < \alpha < \omega_{1}$. Then:
\begin{enumerate}[\upshape (1)]
    \item A set $\{ f_{i} \colon i \in I \}$ of functions $f_{i} \colon \omega \to \mathbb{G}$ is $p$-independent if and only if the function 
    \[
      \tilde{f}\colon I \to \mathsf{ult}_{p}^{1}(\mathbb{G})/\mathsf{ult}_{p}^{0}(\mathbb{G})
    \]
    defined by $\tilde{f}(i)=\pi_{0}^{1}([f_{i}])$ for $i \in I$ is linearly independent, where $\pi_{0}^{1} \colon \mathsf{ult}_{p}^{1}(\mathbb{G}) \to \mathsf{ult}_{p}^{1}(\mathbb{G})/\mathsf{ult}_{p}^{0}(\mathbb{G})$ denotes the natural projection.
    \item A set $\{ f_{i} \colon i \in I \}$ of functions $f_{i} \colon \omega \to \mathsf{ult}_{p}^{\alpha}(\mathbb{G})$ is $p$-independent if and only if the set $\{\tilde{f}_{i} \colon i \in I\}$ of  functions $\tilde{f}_{i} \colon \omega \to \mathsf{ult}_{p}^{\alpha}(\mathbb{G})/\mathsf{ult}_{p}^{\alpha^{-}}(\mathbb{G})$ is a $p$-independent set, where each $\tilde{f}_{i}$ is defined by 
    $\tilde{f}_{i}(n)=\pi_{\alpha^{-}}^{\alpha}(f_{i}(n))$ for $n \in \omega$ and   
    \[
      \pi^{\alpha}_{\alpha^{-}} \colon \mathsf{ult}_{p}^{\alpha}(\mathbb{G}) \to \mathsf{ult}_{p}^{\alpha}(\mathbb{G})/ \mathsf{ult}_{p}^{\alpha^{-}}(\mathbb{G})
    \]
    denotes the natural projection.
\end{enumerate}
\end{lemma}

\begin{proof}
To see (1), note that  
\[
\sum_{i \in E}[f_{i}]=[\langle a \rangle]
\]
iff 
\[
\left\{n \colon a + \sum_{i \in E}f_{i}(n) = \varnothing\right\} \in p,  
\]
for every non-empty finite set $E \subset I$ and every $a \in \mathbb{G}$.

\medskip

To see (2). Let $E \subseteq I$ be a non-empty finite set and $a \in \mathsf{ult}_{p}^{\alpha}(\mathbb{G})$ and, notice that 
\[
\left\{n \colon \sum_{i \in E}\tilde{f}_{i}(n) = \pi^{\alpha}_{\alpha^{-}}(a) \right\} \in p
\]
iff 
\[
  \left\{ n \colon a + \sum_{i \in E} f_{i}(n) \in \mathsf{ult}_{p}^{\alpha^{-}}(\mathbb{G})\right\}\in p
\]
iff 
\[
 \left\{ n \colon (a + [f])+ \sum_{i \in E} f_{i}(n)=\varnothing\right\}\in p,
\]
where for some $U \in p$ we have that $f(n)=a + \sum_{i \in E} f_{i}(n)\in \mathsf{ult}_{p}^{\alpha^{-}}(\mathbb{G})$ for $n \in U$.
\end{proof}

\medskip

Note also that if $\text{ht}([f])=\alpha$ for $\alpha > 0$, then $f$ is not constant on an element of $p$ (equivalently, $\{f\}$ is $p$-independent).

\medskip

\begin{lemma}\label{L:main}
 Let $0 < \alpha < \omega_{1}$, $[f] \in \mathsf{ult}_{p}^{\alpha}([\omega]^{<\omega})$ and $p$ a selective ultrafilter. If $f$ is not constant on an element of $p$, then there is a tree $T \in \mathbb{L}_{p}(T_{\alpha})$ with $T \subseteq T_{f}$ such that $\restr{\hat{f}}{\Omega_{0}(T)}$ is linearly independent.\footnote{Here, we are using the notation from the proof of Proposition \ref{p-compact} (1).}
\end{lemma}

\begin{proof}
First, if $\alpha = 1$, then the conclusion of the lemma follows from Proposition \ref{selective} (2) (b). Thus, we may assume that $\alpha \geqslant 2$. 

\medskip

We plan to construct a tree $T \in \mathbb{L}_{p}(T_{\alpha})$ with $T \subseteq T_{f}$, so that the following hold for any $\beta \leqslant \alpha$:

\begin{itemize}
    \item if $\beta > 0$, then $\langle \hat{f}(t) \colon t \in \Omega_{\beta}(T)\rangle$ forms a $p$-independence sequence;
    \item if $\beta = 0$, then $\langle \hat{f}(t) \colon t \in \Omega_{0}(T)\rangle$ forms a linearly independent sequence.
\end{itemize}

\medskip
In order to do this, first, we recursively construct a tree $T^{*} \in \mathbb{L}_{p}(T_{\alpha})$ with $T^{*} \subseteq T_{f}$, so that the following hold for any $t \in T^{*}$ with $\rho_{T^{*}}(t) \geqslant 1$:

\begin{itemize}
    \item if $\text{ht}(\hat{f}(t))=1$, then $\langle \hat{f}(t^{\frown}n) \colon n \in \text{succ}_{T^{*}}(t)\rangle \subset [\omega]^{<\omega}$ forms a linearly independent sequence;
    \item if $\text{ht}(\hat{f}(t))=\beta + 1$ with $\beta \geqslant 1$, then 
    $\langle \hat{f}(t^{\frown}n) \colon n \in \text{succ}_{T^{*}}(t)\rangle \subset \mathsf{ult}_{p}^{\beta}([\omega]^{<\omega})$ forms a $p$-independent sequence;
    \item if $\text{ht}(\hat{f}(t))$ is a limit ordinal, then $\langle \text{ht}(\hat{f}(t^{\frown}n)) \colon n \in \text{succ}_{T^{*}}(t)\rangle$ is a strictly increasing sequence of non-zero ordinals.
\end{itemize}

\medskip

At step $t$. If $\text{ht}(\hat{f}(t))=1$ and $\langle \hat{f}(t^{\frown}n) \colon n \in \text{succ}_{T_{f}}(t)\rangle$ is not constant on an element of $p$, then $\rho_{T_{f}}(t)=1$ and  applying Proposition \ref{selective} (2) (b) there exists $U \in p$ with $U \subseteq \text{succ}_{T_{f}}(t)$ such that $\langle \hat{f}(t^{\frown}n) \colon n \in U\rangle$ is linearly independent. Therefore, in this case we put $\text{succ}_{T^{*}}(t)=U$.

\medskip

If $\text{ht}(\hat{f}(t))=\beta + 1$ with $\beta \geqslant 1$ and $\langle \hat{f}(t^{\frown}n) \colon n \in \text{succ}_{T_{f}}(t)\rangle$ is not constant on an element of $p$, then consider the sequence 
\[
\tilde{f}_{t} \colon \text{succ}_{T_{f}}(t) \to \mathsf{ult}_{p}^{\beta}([\omega]^{<\omega})/ \mathsf{ult}_{p}^{-\beta}([\omega]^{<\omega})
\]
defined by  $\tilde{f}_{t}(n)= \pi^{\beta}_{\beta^{-}}(\hat{f}(t^{\frown}n))$ for $n \in \text{succ}_{T_{f}}(t)$. Since $\langle \hat{f}(t^{\frown}n) \colon n \in \text{succ}_{T_{f}}(t)\rangle$ is not constant on an element of $p$, by Lemma \ref{L:p-independence} (2), the sequence $\tilde{f}_{t}$ is not constant on an element of $p$. Therefore, applying Proposition \ref{selective} (2) (b) and Remark \ref{R:selective}, we can find an element $U \in p$ with $U \subseteq \text{succ}_{T_{f}}(t)$ such that 
$\restr{\tilde{f}_{t}}{U}$ is linearly independent. Thus, by Lemma \ref{L:p-independence} (1), putting $\text{succ}_{T^{*}}(t)=U$ we can conclude that $\langle \hat{f}(t^{\frown}n) \colon n \in \text{succ}_{T^{*}}(t)\rangle$ forms a $p$-independent sequence.

\medskip

If $\text{ht}(\hat{f}(t))=\beta$ is a limit ordinal, then for every $\delta < \beta$ we set  
$U_{\delta}=\{n \in \text{succ}_{T_{f}}(t) \colon \text{ht}(\hat{f}(t^{\frown}n))=\delta\}$. Then 
\[
\bigsqcup_{\delta < \beta}U_{\delta}=\text{succ}_{T_{f}}(t),
\]
where each $U_{\delta} \notin p$. The selectiveness of $p$ implies that there is an $U \in p$ such that $|U \cap U_{\delta}|\leqslant 1$ for every $\delta < \beta$. Thus, in this case put $\text{succ}_{T^{*}}(t)=U \setminus U_{0}$. This concludes recursive construction of $T^{*}$.

\medskip
Notice that $\rho_{T^{*}}(t)=\text{ht}(\hat{f}(t))$ for every $t \in T^{*}$. Now given a tree $T^{\prime} \in \mathbb{L}_{p}(T_{\alpha})$ with $T^{\prime} \subseteq T^{*}$, we can canonically list its members $t^{\prime} \in T^{\prime}$ as $\{t_{k}^{T^{\prime}} \colon k <\omega\}$ so that  

\begin{itemize}
    \item $t_{k}^{T^{\prime}} \subset t_{l}^{T^{\prime}}$ entails 
    $k < l$; 
    \item $t_{k}^{T^{\prime}}=t^{\frown} n$, $t_{l}^{T^{\prime}}=t^{\frown} m$, $\text{ht}(\hat{f}(t))$ is a limit ordinal, and $\text{ht}(\hat{f}(t^{\frown} n))<\text{ht}(\hat{f}(t^{\frown} m))$ entails $k < l$;
    \item $t_{k}^{T^{\prime}}=t^{\frown} n$, $t_{l}^{T^{\prime}}=t^{\frown} m$, $\text{ht}(\hat{f}(t))$ is a successor ordinal, and $n < m$ entails $k < l$.
\end{itemize}

Choose a sufficiently large regular cardinal $\theta$ and a countable elementary submodel $M$ of $\langle H(\theta), \in \rangle$ containing all the relevant objects such as $p$ and $T^{*}$. Fix $U \in p$ so  that $U$ is a pseudo-intersection of $p \cap M$. Put  $T^{**}=\restr{T^{*}}{U}$ and $V_{t} = \text{succ}_{T^{**}}(t)$ for $t \in (T^{**})^{+}$. 

\medskip

We unfix $t$, and construct by recursion on $k$ the required condition $T=\{t_{k}^{T} \colon k \in \omega\} \in \mathbb{L}_{p}(T_{\alpha})$ with $T \subseteq T^{**}$, as well as an auxiliary function $g \colon T^{+} \to \omega$ and sets $W_{t} \subseteq V_{t}$ for $t \in T^{+}$ such that the following are satisfied: 

\begin{enumerate}[\upshape (a)]
    \item $W_{t}= V_{t} \setminus g(t)=\text{succ}_{T}(t)$ for all $t \in T^{+}$ (by definition).
    \item For all $k$, 
      \begin{itemize}
          \item if $\rho_{T}(t^{T}_{k})= 1$, then 
          \[
           \left\langle \hat{f}(t_{l}^{T}{}^{\frown}n) \colon \exists \, l \leqslant k \left( n \in W_{t_{l}^{T}} \ \& \ \rho_{T}(t_{l}^{T}{}^{\frown}n)= 0 \right)\right\rangle \subseteq [\omega]^{<\omega}
          \]
          forms a linearly independent sequence;
          \item if $\rho_{T}(t^{T}_{k})= \beta +1$ with $\beta \geqslant 1$, then
          \[
          \left\langle \hat{f}(t_{l}^{T}{}^{\frown}n) \colon \exists \, l \leqslant k \left( n \in W_{t_{l}^{T}} \ \& \ \rho_{T}(t_{l}^{T}{}^{\frown}n)=\beta\right)\right\rangle \subset \mathsf{ult}_{p}^{\beta}([\omega]^{<\omega})
          \]
          forms a $p$-independence sequence; 
          \item if $\rho_{T}(t^{T}_{k})=\beta$ is a limit ordinal, then 
            \[
            \left\langle \text{ht}(\hat{f}(t_{l}^{T}{}^{\frown}n)) \colon \exists \, l \leqslant k          \left( n \in W_{t_{l}^{T}} \ \& \          \rho_{T}(t_{l}^{T})=\beta\right)\right\rangle
            \]
            forms an one-to-one sequence, and
            \begin{align*}
            &\sup\left\{\text{ht}(\hat{f}(t_{l}^{T}{}^{\frown}n)) \colon \exists \, l < k \left( \rho_{T}(t_{l}^{T}) \ne \beta \ \& \ n \in W_{t_{l}^{T}} \ \& \ \rho_{T}(t_{l}^{T}{}^{\frown}n) < \beta \right)\right\}\\ &< 
            \min\left\{\text{ht}(\hat{f}(t_{k}^{T}{}^{\frown}n)) \colon n \in W_{t_{k}^{T}}\right\}.
            \end{align*}
      \end{itemize}
\end{enumerate}

Before describing the construction let us recall a simple fact from linear algebra:

\begin{fact}\label{F:linearly-independent}
Let $A$ and $B$ be linearly independent sets in a Boolean group with $A$ a finite set. Then there is $A'\subseteq B$ such that $|A'|\leq |A|$ and $A\sqcup (B\setminus A')$ is linearly independent.
\end{fact}

\begin{prooffact}
Let $V_{A}=\text{span}(A)$ and $V_{B}=\text{span}(B)$. Then $\text{dim}(V_{A} \cap V_{B}) \leqslant |A|$, so there exists a set $A^{\prime}\subseteq B$ such that $\text{span}(A^{\prime})=V_{A} \cap V_{B}$. Therefore, 
$|A^{\prime}|\leqslant |A|$ and $A \sqcup (B\setminus A^{\prime})$ is linearly independent.
\end{prooffact}

\medskip
\textsf{Basic step} $k=0$. So $t_{0}^{T}=\varnothing$. We put $g(t_{0}^{T})=0$ and hence $W_{t_{0}^{T}}=V_{t_{0}^{T}}$. The conditions (a) and (b) are immediate. 

\medskip
\textsf{Recursion step} $k > 0$. Assume $W_{t^{T}_{l}}$ (for $l<k$) as well as $\restr{g}{k}$ have been defined so as to satisfy (a) and (b). In particular, we know already $t_{k}^{T}$, for it is of the form $t_{l}^{T}{}^{\frown}n$ for some 
$n \in W_{t_{l}^{T}}$ where $l < k$. Put $\rho_{T}(t^{T}_{k})=\gamma$ and assume $\gamma \geqslant 1$. Note that, since (b) is satisfied for $l$, we must have $\rho_{T}(t_{l}^{T})=\gamma + 1$ and  
\[
\left\langle \hat{f}(t_{j}^{T}{}^{\frown}m) \colon \exists \, j \leqslant l \left( m \in W_{t_{j}^{T}} \ \& \ \rho_{T}(t_{j}^{T}{}^{\frown}m) = \gamma \right)\right\rangle \subset \mathsf{ult}_{p}^{\gamma}([\omega]^{<\omega})
\]
is a $p$-independent sequence. Put
\begin{align*}
A_{l} &= \{t_{l^{\prime}}^{T} \colon l^{\prime} \leqslant k \ \& \ \rho_{T}(t_{l^{\prime}}^{T}) = \gamma\} \\
&\subset \left\{t_{j}^{T}{}^{\frown}m \colon \exists \, j \leqslant l \left( m \in W_{t_{j}^{T}} \ \& \ \rho_{T}(t_{j}^{T}{}^{\frown}m)= \gamma \right) \right\}
\end{align*}
and $ A_{l}^{-}=A_{l}\setminus \{t_{k}^{T}\}$. 

\medskip

If $\gamma= 1$, then applying Proposition \ref{selective} (2) (c) there exists $V \in p$ and a function $g_{l} \colon A_{l} \to \omega$ such that
\[
\left \langle \hat{f}(t^{\frown}m) \colon t \in A_{l} \ \& \ m \in V \setminus g_{l}(t) \right \rangle \subseteq [\omega]^{<\omega}
\]
is a linearly independent sequence. Using the elementarity of $M$ and our assumption about $U$ we conclude that there exists a function $g_{l,U} \colon A_{l} \to \omega$ such that
\[
\left \langle \hat{f}(t^{\frown}m) \colon t \in A_{l} \ \& \ m \in U \setminus g_{l,U}(t) \right \rangle \subseteq [\omega]^{<\omega}
\]
is a linearly independent sequence. Note that $V_{t_{k}^{T}} \setminus g_{l,U}(t_{k}^{T}) \subseteq U \setminus g_{l,U}(t_{k}^{T})$ and $W_{t} \setminus g_{l,U}(t) \subseteq U \setminus g_{l,U}(t)$ for  $t \in A_{l}^{-}$. Since $A_{l}$ is a finite set, using Fact \ref{F:linearly-independent}, we can find a natural number $g(t_{k}^{T}) \geqslant g_{l,U}(t_{k}^{T})$ so that 
\[
 \left \langle \hat{f}(t^{\frown}m) \colon t \in A_{l}^{-} \ \& \ m \in W_{t} \right \rangle \cup \left \langle \hat{f}(t_{k}^{T}{}^{\frown}m) \colon m \in V_{t_{k}^{T}} \setminus g(t_{k}^{T}) \right \rangle
\]
forms a linearly independent sequence, as required.

\medskip

For the case $\gamma= \beta +1$ with $\beta \geqslant 1$, we will proceed in a similar way as the previous case. Given $t\in A_{l}$, let 
\[
\tilde{f}_{t} \colon V_{t} \to \mathsf{ult}_{p}^{\beta}([\omega]^{<\omega})/\mathsf{ult}_{p}^{\beta^{-}}([\omega]^{<\omega})
\]
be defined by $\tilde{f}_{t}(m)=
\pi_{\beta}^{\beta^{-}}(\hat{f}(t^{\frown}m))$ for $m \in V_{t}$. By Lemma \ref{L:p-independence} (2), $\{\tilde{f}_{t} \colon t \in A_{l}\}$ is a $p$-independent set. Thus, applying Proposition \ref{selective} (2) (c) and Remark \ref{R:selective}, we can find an element $V \in p$ and a function $g_{l} \colon A_{l} \to \omega$ such that
\[
\left \langle \tilde{f}_{t}(m) \colon t \in A_{l} \ \& \ m \in V \setminus g_{l}(t) \right \rangle \subseteq \mathsf{ult}_{p}^{\beta}([\omega]^{<\omega})/\mathsf{ult}_{p}^{\beta^{-}}([\omega]^{<\omega})
\]
is a linearly independent sequence. By elementarity of $M$ and the property of $U$ we have that there exists a function $g_{l,U} \colon A_{l} \to \omega$ such that
\[
\left \langle \tilde{f}_{t}(m) \colon t \in A_{l} \ \& \ m \in U \setminus g_{l,U}(t) \right \rangle
\]
is a linearly independent sequence. Since $A_{l}$ is a finite set, $V_{t_{k}^{T}} \setminus g_{l,U}(t_{k}^{T}) \subseteq U \setminus g_{l,U}(t_{k}^{T})$ and $W_{t} \setminus g_{l,U}(t) \subseteq U \setminus g_{l,U}(t)$ for  $t \in A_{l}^{-}$, using Fact \ref{F:linearly-independent}, we can find a natural number $g(t_{k}^{T}) \geqslant g_{l,U}(t_{k}^{T})$ so that
\[
 \left \langle \tilde{f}_{t}(m) \colon t \in A_{l}^{-} \ \& \ m \in W_{t} \right \rangle \cup \left \langle \tilde{f}_{t_{k}^{T}}(m) \colon m \in V_{t_{k}^{T}} \setminus g(t_{k}^{T}) \right \rangle
\]
forms a linearly independent sequence and, by Lemma \ref{L:p-independence} (1), this means that 
\[
 \left \langle \hat{f}(t^{\frown}m) \colon t \in A_{l}^{-} \ \& \ m \in W_{t} \right \rangle \cup \left \langle \hat{f}(t_{k}^{T}{}^{\frown}m) \colon m \in V_{t_{k}^{T}} \setminus g(t_{k}^{T}) \right \rangle \subset \mathsf{ult}_{p}^{\beta}([\omega]^{<\omega})
\]
forms a $p$-independent sequence, as required.

\medskip

If $\gamma$ is a limit ordinal, then applying Proposition \ref{selective} (2) (c) there exists $V \in p$ and a function $g_{l} \colon A_{l} \to \omega$ such that
\[
\left \langle \hat{f}(t^{\frown}m) \colon t \in A_{l} \ \& \ m \in V \setminus g_{l}(t) \right \rangle \subset \mathsf{ult}_{p}^{\gamma^{-}}([\omega]^{<\omega})
\]
is a linearly independent sequence. Thus, proceeding as previous cases, it is possible to find a function $g_{l,U} \colon A_{l} \to \omega$ and a natural number $g(t_{k}^{T}) \geqslant g_{l,U}(t_{k}^{T})$ so that 
\[
 \left \langle \hat{f}(t^{\frown}m) \colon t \in A_{l}^{-} \ \& \ m \in W_{t} \right \rangle \cup \left \langle \hat{f}(t_{k}^{T}{}^{\frown}m) \colon m \in V_{t_{k}^{T}} \setminus g(t_{k}^{T}) \right \rangle
\]
forms a linearly independent sequence. In particular, 
\[
\left \langle \text{ht}(\hat{f}(t^{\frown}m)) \colon t \in A_{l}^{-} \ \& \ m \in W_{t} \right \rangle \cup \left \langle \text{ht}(\hat{f}(t_{k}^{T}{}^{\frown}m)) \colon m \in V_{t_{k}^{T}} \setminus g(t_{k}^{T}) \right \rangle
\]
forms an one-to-one sequence and, since $\gamma$ is a limit ordinal, one sees that without loss of generality, we may assume that
\begin{align*}
  &\sup\left\{\text{ht}(\hat{f}(t_{l}^{T}{}^{\frown}m)) \colon \exists \, l < k \left( \rho_{T}(t_{l}^{T}) \ne \gamma \ \& \ m \in W_{t_{l}^{T}} \ \& \ \rho_{T}(t_{l}^{T}{}^{\frown}m) < \gamma \right)\right\}\\ &< 
  \min\left\{\text{ht}(\hat{f}(t_{k}^{T}{}^{\frown}m)) \colon m \in V_{t_{k}^{T}} \setminus g(t_{k}^{T})\right\},
\end{align*}            
as required.
\end{proof}

\medskip

Now we are ready to prove the main theorem of this section.

\medskip

\begin{proofmain}{\ref{Th:SelectiveUltrafilter}}
According to Proposition \ref{p-compact}, $\mathsf{Ult}_{p}^{\omega_{1}}([\omega]^{<\omega})$ is a Hausdorff $p$-compact topological group. It remains therefore only to show that $\mathsf{Ult}_{p}^{\omega_{1}}([\omega]^{<\omega})$ contains no non-trivial convergent sequences to $\pi([\langle \varnothing\rangle])$. To see this, let $\tilde{f} \colon \omega \to \mathsf{Ult}_{p}^{\omega_{1}}([\omega]^{<\omega})$ be a non-trivial sequence, say $\tilde{f}(n)=\pi(f(n))$ ($n \in \omega$) where $f \colon \omega \to \mathsf{ult}_{p}^{\omega_{1}}([\omega]^{<\omega})$. Without loss of generality we can assume that $\tilde{f}$ is a one-to-one function. Thus, since  
\[
\mathsf{ult}_{p}^{\omega_{1}}([\omega]^{<\omega})=\mathsf{ult}_{p}\left(\bigcup_{\alpha < \omega_{1}}\mathsf{ult}_{p}^{\alpha}([\omega]^{<\omega})\right),
\]
there exists $0 < \alpha < \omega_{1}$ so that $[f] \in \mathsf{ult}_{p}^{\alpha}([\omega]^{<\omega})$ and $f$ is not constant on an element of $p$. By Lemma \ref{L:main}, there is a tree $T \in \mathbb{L}_{p}(T_{\alpha})$ with $T \subseteq T_{f}$ such that $\restr{\hat{f}}{\Omega_{0}(T)}$ is linearly independent. Note that $\hat{f}[\Omega_{0}(T)] \subseteq [\omega]^{<\omega}$. Take $\Phi \in \text{Hom}([\omega]^{<\omega},2)$ so that 
$\hat{f}[\Omega_{0}(T)] \subseteq \Phi^{-1}(1)$. So $\overline{\Phi}([\hat{f}])=1$ and hence $\overline{\Phi}([f])=1$. Thus, $\overline{\Phi}$ is a witness that the sequence $f$ does not $\tau_{\, \overline{\text{Bohr}}}$-converge to $[\langle \varnothing\rangle]$ and, since $\tilde{f}$ is one-to-one, in fact  $\tilde{f}$ does not converge to $\pi([\langle \varnothing\rangle])$.
\end{proofmain}

\medskip

\section{Countably compact group without convergent sequences} \label{cc-group}

In this section we develop the ideas introduced in the previous section into a {\sf ZFC} construction of a countably compact subgroup of $2^\mathfrak c$ without non-trivial convergent sequences. Recall that any boolean group of size $\mathfrak c$ (in particular $\mathsf{Ult}_{p}^{\omega_{1}}([\omega]^{<\omega})$) is isomorphic to $[\mathfrak c]^{<\omega}$. In fact, the extension of homomorphisms produces a (topological and algebraic) embedding $h$ of $(\mathsf{Ult}_{p}^{\omega_{1}}([\omega]^{<\omega}), \tau_{\, \overline{\text{Bohr}}})$ into $2
^\mathfrak c \simeq 2^{\text{Hom}([\omega]^{<\omega}, 2)}$ defined by
$$h([f])(\Phi)= \overline\Phi ([f]).$$

Similarly to the ultrapower construction, we shall extend the Bohr topology $\tau_{\, \text{Bohr}}$ on $[\omega]^{<\omega}$ to a group topology $\tau_{\, \overline{\text{Bohr}}}$ on $[\mathfrak c]^{<\omega}$ to obtain the result. The difference is that rather than using a single ultrafilter, we shall use a carefully constructed $\mathfrak c$-sized family of ultrafilters.

\begin{theorem}\label{Th: main} There is a Hausdorff countably compact topological Boolean group without non-trivial convergent sequences.
\end{theorem}

\begin{proof}
We shall construct a countably compact topology on $[\mathfrak c]^{<\omega}$ starting from $([\omega]^{<\omega}, \tau_{\, \text{Bohr}})$ as follows:

\medskip

Fix  an indexed family $\{f_\alpha \colon \alpha\in [\omega,\mathfrak c)\}\subset  ([\mathfrak c]^{<\omega})^\omega$ of one-to-one sequences such that
\begin{enumerate}
 \item for every infinite $X\subseteq[\mathfrak c]^{<\omega}$ there is an 
       $\alpha\in [\omega,\mathfrak c)$ with $\text{rng}(f_\alpha)\subseteq X$,
 \item each $f_\alpha$ is a sequence of linearly independent elements, and
 \item $\text{rng}(f_\alpha)\subset [\alpha]^{<\omega}$ for every 
       $\alpha\in [\omega,\mathfrak c)$. 
\end{enumerate}

\medskip

Given a sequence 
$\{p_\alpha \colon \alpha\in [\omega,\mathfrak c)\} \subset \omega^*$ define for every $\Phi\in \text{Hom}([\omega]^{<\omega},2)$ its extension $\overline{\Phi}\in  \text{Hom}([\mathfrak c]^{<\omega},2)$ recursively by putting 
$$\overline{\Phi}(\{\alpha\})=p_\alpha\text{-}\lim_{n \in \omega} \overline \Phi (f_\alpha(n)).$$
%
%Note that doing this indeed defines unique extension of $\Phi $ to a homomorphism on $[\mathfrak c]^{<\omega}$ to $2$, which, moreover, has the property that $\overline{\Phi}(\{\alpha\})=p_\alpha\text{-}\lim_{n \in \omega} \overline \Phi (f_\alpha(n))$ for every $\Phi$ and every $\alpha\in [\omega,\mathfrak c)$.

Note that $[\omega]^{<\omega}$ together with the independent set $\{\{\alpha\}: \alpha \in [\omega,\mathfrak c)\}$ generate the group $[\mathfrak c]^{<\omega}$ so the above definition uniquely extends $\Phi$ to a homomorphism $\overline \Phi: [\mathfrak c]^{<\omega}\to 2$.

\medskip

This allows us to define the topology $\tau_{\, \overline{\text{Bohr}}}$ induced by $\{\overline{\Phi}: \Phi \in \text{Hom}([\omega]^{<\omega},2)\}$ on $[\mathfrak c]^{<\omega}$ as the weakest topology making all 
$\overline{\Phi}$ continuous (for $\Phi \in \text{Hom}([\omega]^{<\omega},2)$), or equivalently, the group topology having $\{\text{Ker}(\overline{\Phi}): \Phi \in \text{Hom}([\omega]^{<\omega},2)\}$ as a subbasis of the filter of neighbourhoods of the neutral element $\varnothing$. It follows directly from the above observation that independently of the choice of the ultrafilters the topology is a countably compact group topology on $[\mathfrak c]^{<\omega}$. Indeed, $\{\alpha\} \in \overline{\{f_{\alpha}(n) \colon n \in \omega\}}^{\tau_{\, \overline{\text{Bohr}}}}$ for every $\alpha\in [\omega,\mathfrak c)$, in fact $\{\alpha\}=p_\alpha\text{-}\lim_{n \in \omega}f_\alpha(n)$.

\medskip

 Call a set $D\in [\mathfrak c]^{\omega}$ \emph{suitably closed} if $\omega\subseteq D$ and $\bigcup_{n\in\omega} f_\alpha(n)\subseteq D$ for every $\alpha\in D$. The following claim shows that  the construction  is  locally countable.

\begin{claim}\label{key} The topology $\tau_{\, \overline{\text{Bohr}}}$ contains no non-trivial convergent sequences if and only if
$\forall D\in [\mathfrak c]^{\omega} \text{ suitably closed } \exists \Psi\in \text{Hom}([D]^{<\omega},2) \text{ such that }$
\begin{enumerate}
\item  $\forall \alpha \in D\setminus\omega \ \Psi (\{\alpha\})=p_\alpha\text{-}\lim_{n \in \omega} \Psi(f_\alpha(n))$;
\item $\forall i\in 2 \ |\{ n: \Psi(f_\alpha (n))=i\}|=\omega$.
\end{enumerate}
\end{claim}

\begin{proofclaim}
 Given an infinite  $X\subseteq [\mathfrak c]^{<\omega}$ there is an 
$\alpha\in [\omega,\mathfrak c)$ such that $\text{rng}(f_\alpha)\subseteq X$. Let $D$ be suitably closed with $\alpha\in D$,  and let $\Psi$ be the given homomorphism.  It follows directly from the definition, and property (1) of $\Psi$, that, if  $\Phi=\Psi \restriction [\omega]^{<\omega}$ then in turn $\Psi=\overline{\Phi} \restriction [D]^{<\omega}$, which implies that  
$\langle f_\alpha(n) \colon n \in \omega \rangle$ (and hence also $X$) is not a convergent sequence as $\overline{\Phi}$ takes both values $0$ and $1$ infinitely often on the set $\{ f_\alpha(n) \colon n \in \omega\}$. 

The reverse implication is even more trivial (and not really necessary for the proof).
\end{proofclaim}

\medskip

Note that if this happens then, in particular,
$$K=\bigcap_{\Phi \in \text{Hom}([\omega]^{<\omega},2)}  \text{Ker}(\overline \Phi)$$
is finite, and $[\mathfrak c]^{<\omega}/K$ with the quotient topology is the Hausdorff countably compact group without non-trivial convergent sequences we want.

\medskip

Hence to finish the proof it suffices to produce a suitable family of ultrafilters.

\begin{claim}\label{ult-fam}
There is a family $\{p_\alpha \colon \alpha < \mathfrak{c}\}$ of  free ultrafilters on $\omega $ such that
for every $D\in [\mathfrak c]^\omega $ and $ \{f_\alpha:\alpha \in D\}$ such that each $f_\alpha$ is an one-to-one enumeration of linearly independent elements of $[\mathfrak c]^{<\omega}$ there is a sequence $\langle U_\alpha:\alpha\in D\rangle$ such that
\begin{enumerate}
\item $\{U_\alpha:\alpha\in D\}$ is a family of pairwise disjoint subsets of $\omega$,
\item $U_\alpha \in p_\alpha$ for every $ \alpha\in D $, and
\item $\{  f_\alpha (n): \alpha\in D \ \& \ n \in U_\alpha\}$ is a linearly independent subset of $[\mathfrak c]^{<\omega}$.
\end{enumerate}
\end{claim}
\begin{proofclaim} 
Fix $\{I_n: n\in\omega\}$ a partition of $\omega$ into finite sets such that 
$$ |I_n|>n \cdot \sum_{m<n}|I_m|,$$
and let 
$$\mathcal B=\{ B\subseteq \omega : \ \forall n\in \omega \  |I_n\setminus B| \leqslant \sum_{m<n}|I_m|\}.$$
Note that $\mathcal B$ is a centered family, and denote by $\mathcal F$ the filter it generates. Note also, that if $A$ is an infinite subset of $\omega$ then $\bigcup_{n \in A}I_n\in \mathcal F^+$.

\medskip

Let $\{ A_\alpha:\alpha\in \mathfrak{c}\}$ be any almost disjoint family of size $\mathfrak c$ of infinite subsets of $\omega$, and let, for every $\alpha<\mathfrak c$,  $p_\alpha$ be any ultrafilter on $\omega$ extending $\mathcal F\restriction 
\bigcup_{n \in A_\alpha}I_n$. 

\medskip

To see that this works, let $D=\{\alpha_n:n\in\omega\}$ and a family $\{f_\alpha:\alpha\in D\}$ of one-to-one sequences of linearly independent elements of $[\mathfrak{c}]^{<\omega}$ be given. Let $\{B_n: n\in\omega\}$ be a partition of $\omega$ such that $B_n=^* A_{\alpha_n}$ for every $n\in\omega$, and recursively define a set $B$ such that, $I_0\subseteq B$,  
$$  |I_n\setminus B| \leqslant  \sum_{m<n}|I_m|$$
for every $n >0$, and
$$\{f_{\alpha_n}(m): m\in B\cap I_l, \, l\in B_n\text{ and }n \in \omega\}\text{ is linearly independent}.$$ 

\smallskip

In order to obtain the set $B$ we recursively  use Fact \ref{F:linearly-independent} to construct a sequence $\{C_{l} \colon l \in \omega\}$ of finite sets such that:

\begin{itemize}
    \item $C_{0}=\varnothing$;
    \item If $l>0$, then 
    \begin{enumerate}[(i)]
        \item $C_{l} \subseteq I_{l}$,
        \item $|C_{l}|\leqslant \sum_{i<l} |I_{i} \setminus C_{i}|$, and
        \item $\bigsqcup_{i<l}f_{\alpha_{n_{i}}}[I_{i}\setminus C_{i}]$ is linearly independent, where $n_{i}$ is such that $i \in B_{n_{i}}$.
    \end{enumerate}
\end{itemize}

Put $B=\bigcup_{l \in \omega}I_{l}\setminus C_{l}$. By (ii), it follows that $B\in \mathcal{B}\subseteq \mathcal{F}$. Since $B_{n}=^{*}A_{n}$ and $B \in \mathcal{F}$, is clear that $U_{n}:=B \cap \bigcup_{l\in  B_{n}}I_{l}\in p_{\alpha_{n}}$. By (iii), it follows that $\{f_{\alpha_{n}}(m) \colon m \in U_{n} \text{ and } n \in \omega\}$ is linearly independent. Therefore, the sequence 
$\langle U_{n} \colon n \in \omega\rangle$ is as required.
\end{proofclaim}

\medskip

Now, use this family of ultrafilters as the parameter  in the construction of the topology described above. By Claim \ref{key} it suffices to show that given a suitably closed $D\subseteq \mathfrak c$ 
and $\alpha\in D\setminus \omega$ there is a homomorphism $\Psi: [D]^{<\omega}\to 2$ such that 
\begin{enumerate}
\item  $\forall \alpha \in D\setminus\omega \ \Psi (\{\alpha\})=p_\alpha\text{-}\lim_{n \in \omega} \Psi(f_\alpha(n))$
\item $\forall i\in 2 \ |\{ n: \Psi(f_\alpha (n))=i\}|=\omega$.
\end{enumerate}
By Claim \ref{ult-fam}, there is a sequence 
$\langle U_\alpha:\alpha\in D\setminus \omega \rangle$ such that
\begin{enumerate}
\item $\{U_\alpha:\alpha\in D\setminus\omega \}$ is a family of pairwise disjoint subsets of $\omega$,
\item $U_\alpha \in p_\alpha$ for every $ \alpha\in D\setminus \omega $, and
\item $\{  f_\alpha (n): \alpha\in D\setminus \omega \ \& \ n \in U_\alpha\}$ is a linearly independent subset of $[\mathfrak c]^{<\omega}$.
\end{enumerate}
Enumerate $D\setminus \omega$ as $\{\alpha_n:n\in\omega\}$ so that $\alpha=\alpha_0$. Recursively define a function $h: \{  f_\alpha (n): \alpha\in D\setminus \omega \ \& \ n \in U_\alpha\}\to 2$ so that
\begin{enumerate}

\item  $h$ takes both values $0$ and $1$ infinitely often on $\{  f_{\alpha_0} (n): \ n \in U_{\alpha_0}\setminus \{\alpha_0\}\}$, 
\item $\Psi_0(\{\alpha_0\})=p_{\alpha_0}$-$\lim_{k\in U_{\alpha_0}} \Psi_0(f_{\alpha_0}(k))$, and
\item if $\{\alpha_n\}$ is in the subgroup generated by $\{  f_{\alpha_m} (n): m<n\ \& \ n \in U_{\alpha_m}\}$ then 
$\Psi_n(\{\alpha_n\})=p_{\alpha_n}\text{-}\lim_{k\in U_{\alpha_n}} \Psi_n(f_{\alpha_n}(k))$, and making sure that
\item  $\Psi_n (\{\alpha_n\})=p_{\alpha_n}\text{-}\lim_{k\in U_{\alpha_n}}\Psi_n( f_\alpha(k)).$
\end{enumerate}
Where $\Psi_n$ is a homomorphism defined on the subgroup generated by 
$$\{  f_{\alpha_m} (n): m<n\ \& \ n \in U_{\alpha_m}\}\cup\{\{\alpha_m\}:m<n\}$$
 extending $h\restriction \{  f_{\alpha_m} (n): m<n\ \& \ n \in U_{\alpha_m}\}$.
Then let $\Psi$ be any homomorphism extending $ \bigcup_{m\in\omega} \Psi_m $. Doing this is straightforward given that the  set 
$$\{  f_\alpha (n): \alpha\in D\setminus \omega \ \& \ n \in U_\alpha\}$$
is linearly independent.

\medskip

Finally, note that if we, for $a\in [\mathfrak c]^{<\omega }$, let
$$H(a)(\Phi)= \overline{\Phi}(a) $$
then $H$ is a continuous homomorphism from $[\mathfrak c]^{<\omega}$ to $2^{\text{Hom}([\omega]^{<\omega}, 2)}$ whose kernel is the same group  $K=\bigcap_{\Phi \in \text{Hom}([\omega]^{<\omega})}  \text{Ker}(\overline \Phi)$,
which defines a homeomorphism (and isomorphism) of $[\mathfrak c]^{<\omega}/K$ onto a subgroup of $2^{\text{Hom}([\omega]^{<\omega})}\simeq 2^\mathfrak c$.
\end{proof}

\section{Concluding remarks and questions}

Even though the results of the paper solve longstanding open problems, they also open up very interesting new research possibilities. In Theorem \ref{Th:SelectiveUltrafilter} we showed that if $p$ is a selective ultrafilter then $\mathsf{Ult}_{p}^{\omega_{1}}([\omega]^{<\omega})$ is a $p$-compact  group without non-trivial convergent sequences. This raises the following two interesting questions, the first of which is the equivalent of van Douwen's problem for $p$-compact groups.

\begin{question} Is there in {\sf ZFC} a Hausdorff $p$-compact topological group without a non-trivial convergent sequence? 
\end{question}

A closely related problem asks how much can the property of being selective be weakened in Theorem \ref{Th:SelectiveUltrafilter}. Recall that by Corollary \ref{Cor:P-point} it is consistent that there is a P-point $p$ for which $\mathsf{Ult}_{p}^{\omega_{1}}([\omega]^{<\omega})$ does contain a non-trivial convergent sequence. On the other hand, $\mathsf{Ult}_{p}^{\omega_{1}}([\omega]^{<\omega})\simeq \mathsf{Ult}_{p^\alpha}^{\omega_{1}}([\omega]^{<\omega})$ for every $\alpha<\omega_1$, so there are consistently non-P-points for which $(\mathsf{Ult}_{p}^{\omega_{1}}([\omega]^{<\omega})$ contains no non-trivial convergent sequences.

\begin{question} Is the existence of an ultrafilter $p$ such that  $\mathsf{Ult}_{p}^{\omega_{1}}([\omega]^{<\omega})$ contains no non-trivial convergent sequences equivalent to the existence of a selective ultrafilter?
\end{question}

\begin{question} Is it consistent with {\sf ZFC } that $\mathsf{Ult}_{p}^{\omega_{1}}([\omega]^{<\omega})$ contains a non-trivial convergent sequence for every ultrafilter $p\in\omega^*$?
\end{question}

Assuming $\mathsf{Ult}_{p}^{\omega_{1}}([\omega]^{<\omega})$ contains no non-trivial convergent sequences, it is easy to construct for every $n\in\omega$ a subgroup $\mathbb H$ of $\mathsf{Ult}_{p}^{\omega_{1}}([\omega]^{<\omega})$, such that $\mathbb H^n$ is countably compact while $\mathbb H^{n+1}$ is not. It should be possible to  modify  the construction in Theorem \ref{Th: main} to construct such groups in {\sf ZFC}. These issues will be dealt with in a separate paper.

\smallskip

Another interesting question is:

\begin{question}
Is it consistent with $\ZFC$ that  there is a Hausdorff countably compact topological group without non-trivial convergent sequences of weight $< \mathfrak{c}$? 
\end{question}

Finally, let us recall a 1955 problem of Wallace:

\begin{question}[Wallace \cite{MR67907}] Is every both-sided cancellative countably compact topological semigroup necessarily a group?
\end{question}

It is well known that a counterexample can be recursively constructed inside of any non-torsion countably compact topological group without non-trivial convergent sequences
\cite{MR1328373, MR1426694}. The fact that we do not know how to modify (in {\sf ZFC}) the construction in Theorem \ref{Th: main} to get a non-torsion example of a countably compact group seems surprising. Also the proof  of Theorem \ref{Th:SelectiveUltrafilter} does not seem to easily generalize to  non-torsion groups. Hence:

\begin{question} Is there, in {\sf ZFC }, a non-torsion countably compact  topological group without  non-trivial convergent sequences?
\end{question}

\begin{question} Assume $p\in\omega^*$ is a selective ultrafilter. Does  $(\mathsf{Ult}_{p}^{\omega_{1}}(\mathbb Z), \tau_{\, \overline{\text{Bohr}}})$ contain no non-trivial convergent sequence?
\end{question}

Here the $\tau_{\, \overline{\text{Bohr}}}$ is defined as before as the weakest topology on 
$\mathsf{ult}_{p}^{\omega_{1}}(\mathbb{Z})$ which makes all extensions of homomorphisms from $\mathbb{Z}$ to $\mathbb{T}$ continuous, and the group 
$\mathsf{Ult}_{p}^{\omega_{1}}(\mathbb{Z})=\mathsf{ult}_{p}^{\omega_{1}}(\mathbb{Z})/K$ with $K$ being the intersection of all kernels of the extended homomorphisms.

\bigskip

{\bf Acknowledments.} The authors would like to thank Alan Dow and Osvaldo Guzm\'an for stimulating conversations. The authors also wish to thank the anonymous referee for a thorough reading of the text and for helpful suggestions.


\bibliographystyle{amsplain}

\bibliography{References}

% agregar MR0104753 (21 #3506) Reviewed 
% Kuzʹminov, V., Alexandrov's hypothesis in the theory of topological groups. (Russian) 
% Dokl. Akad. Nauk SSSR 125 1959 727–729. 

\end{document}